\documentclass[10pt]{amsart}
\usepackage{graphicx}
\usepackage{amsthm,amsfonts,amssymb,amsmath}
\usepackage[centerlast]{subfigure}
\usepackage{tikz}
\usepackage{verbatim}

\DeclareMathOperator{\Span}{Span}

\DeclareMathOperator{\ori}{or}
\DeclareMathOperator{\ord}{ord}

\DeclareMathOperator{\surf}{Surf}

\newtheorem{thm}{Theorem}[section]
\newtheorem{lem}[thm]{Lemma}

\newtheorem{prop}[thm]{Proposition}

\theoremstyle{definition}
\newtheorem{df}[thm]{Definition}
\newtheorem{ex}[thm]{Example}

\theoremstyle{remark}
\newtheorem{rem}[thm]{Remark}

\def\ol#1{\overline{#1}}
\def\ul#1{\underline{#1}}

\def\N{\mathbb N}
\def\Z{\mathbb Z}

\def\R{\mathbb R}
\def\C{\mathbb C}
\def\S{\mathbb S}

\def\FS{\mathfrak S}

\def\M{\mathcal M}

\def\Md{\M^{dec}}

\def\MU#1#2{\ul{\M}_{#1,#2}}
\def\MO#1#2{\ol{\M}_{#1,#2}}

\def\MUC#1#2{\ul{\M}_{#1,#2}^{comb}}

\def\MUD#1#2{\ul{\M}_{#1,#2}^{dec}}
\def\MOD#1#2{\ol{\M}_{#1,#2}^{dec}}

\begin{document}

\title[Semistable Graph Homology]{Semistable Graph Homology}
\address {Universidad del Pac\'{i}fico\\
Av. Salaverry 2020, Jes\'{u}s Mar\'{i}a, Lima 11 - Per\'{u}}
\author[J. Z\'{u}\~{n}iga]{Javier Z\'{u}\~{n}iga}
\email{zuniga\underline{ }jj@up.edu.pe}
\date{\today}

\begin{abstract}
Using the orbicell decomposition of the Deligne-Mumford compactification of the moduli space of Riemann surfaces studied in \cite{zun}, a chain complex based on semistable ribbon graphs is constructed which is an extension of Konsevich's graph homology in \cite{kon:fncsg}. 
\end{abstract}

\maketitle

\tableofcontents

\section{Introduction}

The moduli space of Riemann surfaces lies at the intersection of many important areas of mathematics such as 
complex analysis, algebraic geometry, algebraic topology and mathematical physics to name a few. 
When the topological type of the surface is required to have genus $g$ and $n$ labeled points, 
the moduli space $\M_{g,n}$ parametrizes isomorphism classes of complex structures on such surface, 
and it is the solution to the so-called Riemann's moduli problem. 
This space, also called the open moduli space, has many rich algebraic, geometric, analytic and topological properties, 
however it is not compact. 
The Deligne-Mumford (DM) compactification $\MO{g}{n}$ enlarges the moduli space by allowing 
Riemann surfaces to degenerate to stable surfaces with (double point) singularities adding new points to the open moduli space. 
By performing a real blow-up of the DM compactification along the locus of singular surfaces one obtains the space $\MU{g}{n}$, which is also compact but now has a topological boundary parametrizing degenerations. 
Geometrically one can think of the new points in the boundary of this space as remembering 
the angle at which a closed geodesic in a Riemann surface degenerates to form a singularity.

Decorating each labeled point in the surface by positive real numbers 
leads to the space $\Md_{g,n} = \M_{g,n} \times \Delta^{n-1}$ in an obvious way, 
here $\Delta^{n-1}$ is the standard $n-1$ dimensional simplex. 
Analytically one can 
think of these decorations as residues of 
certain quadratic differentials on the surface, as in \cite{strebel}, 
or as the lengths of certain horocycles, as in \cite{pen87}. 
This notion extends to the compactifications $\MOD{g}{n}$ and $\MUD{g}{n}$. 
In the case of the space $\MUD{g}{n}$ its topology is that of $\MU{g}{n} \times \Delta^{n-1}$. 
However, in the case of $\MOD{g}{n}$ its topology is not easy to describe. 
The difference between the homeomorphism types of $\MOD{g}{n}$ and $\MO{g}{n} \times \Delta^{n-1}$ are 
certain toric singularities as pointed out in \cite{zun}. 
However they are still homotopy equivalent.

The advantage of working with decorated spaces is that they have a natural orbicell decomposition via ribbon graphs. 
Kontsevich uses this decomposition to give a proof of Witten's conjecture in \cite{kon:itma}, 
and also shows in \cite{kon:fncsg} how it relates to $A_\infty$-algebras. 
Costello uses a dual version in \cite{cos:pftft} to construct open-closed topological conformal field theories. 
In fact, Kontsevich uses a partial compactification of $\Md_{g,n}$ in \cite{kon:itma}, 
further studied by Looijenga in \cite{loo} and Zvonkine in \cite{zvon}. 
They use the notion of stable ribbon graph. 
This partial compactification misses higher codimension strata of $\MO{g}{n}$. 
The definition of $\MOD{g}{n}$ given in \cite{zun} captures all the strata of $\MO{g}{n}$, 
and by introducing $\MUD{g}{n}$ further  
applications can be developed. 
On the other side, 
the notion of graph that is necessary to deal with $\MUD{g}{n}$ is that of a semistable ribbon graph. These new graphs account for a behavior similar to the ``bubbling'' phenomenon in J-holomorphic curves.

The orbicell decomposition of $\MUD{g}{n}$ leads naturally to the construction of a chain complex 
(with an appropriate notion of orientation) 
which can be used for example 
to give a combinatorial solution to the Quantum Master Equation (QME). 
This equation has been studied from the point of view of string theory by Zwiebach in \cite{zwi:csft}, 
and further formalized by Costello in \cite{cos:pftft} in its closed version. 
For open-closed string theory Zwiebach presents a solution in \cite{zwi:oocst} that is formalized in \cite{hvz}. 
In all these cases a BV-algebra is constructed based on singular or geometric chains 
and a solution to the QME is given using fundamental classes. 
This work represents a first step towards a purely combinatorial solution to the QME 
using an extension of 
ribbon graph homology of \cite{kon:fncsg} by means of the theory developed in \cite{zun}. 
This is expected due to the isomorphism $\MUD{g}{n} \cong \MUC{g}{n}$, 
where $\MUC{g}{n}$ denotes the moduli space of semistable ribbon graphs. 
All these solutions are different geometric incarnations of a purely algebraic phenomenon 
that gives rise to the solution of the QME in the BV formalism studied in \cite{kwz}. In a similar way it is posible to imagine a generalization of $\MUC{g}{n}$ to the open-closed case that can give a combinatorial solution to the QME in \cite{hvz}. The semistable ribbon graphs in this more general case have to account for the boundaries and labeled points in those boundaries on the associated bordered Riemann surfaces.
\bigskip 

\section{Semistable Ribbon Graphs}

The following is a collection of definitions taken from \cite{zun} 
in order to describe the notion of semistable ribbon graph, central to this work. 

\subsection{Ribbon Graphs}

By a graph we mean a combinatorial object consisting of vertices, edges that split into half-edges and incidence relations. 
We avoid isolated vertices. This object is the same as a one dimensional CW-complex up to cellular homeomorphism. 

We will need to consider a special graph homeomorphic to the circle $\S^1$ with only one edge and no vertices. We call this a {\bf semistable circle}. 
The following definition of ribbon graph allows then for the possibility of having multiple connected components, 
some of them possibly semistable circles.
\medskip 

A {\bf ribbon graph} $\Gamma$ is a finite graph together with a 
cyclic ordering on each set of adjacent half-edges to 
every vertex.
\medskip

If $H$ is the set of half-edges and $v$ is a vertex of $\Gamma$, 
let $H_v$ be the set of half-edges adjacent to this vertex. 
The {\bf valence} of a vertex $v$ is then $|H_v|$, the cardinality of $H_v$. 
A {\bf trivalent} graph is one for which all vertices have valence three. 
A cyclic ordering at a vertex $v$ is an ordering of $H_v$ up to cyclic permutation. 
Once a cyclic ordering of $H_v$ is chosen, a cyclic permutation of $H_v$ is defined (an element of $\FS_{H_v}$): 
it moves a half-edge to the next in the cyclic order. 
Define by $\sigma_0$ the element of $\FS_H$ (the group of permutations acting on $H$) 
which is the product of all the cyclic permutations at every vertex. Also, let $\sigma_1$ be the involution 
that interchanges the two half-edges on each edge of $\Gamma$. 
Notice that $\sigma_0$ does not act on half-edges belonging to semistable circles (as semistable circles have no vertices). 
This combinatorial data completely defines the ribbon graph. 

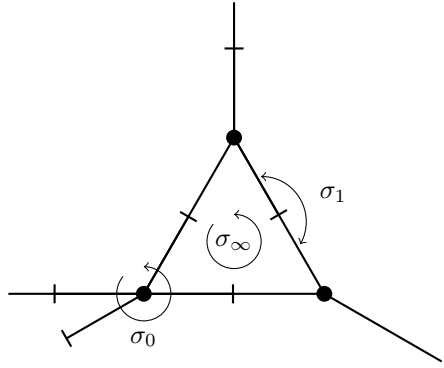
\begin{figure} 
\begin{center}
\begin{tikzpicture}[scale=2.4]
\draw[thick] (1,0) -- (0,0) -- (0.5,0.866) -- (1,0);
\draw[thick] (1,0) -- (1.65,-0.75/2);
\draw[thick] (0.5,0.866) -- (0.5,0.866+0.75); 
\draw[thick] (0,0) -- (-0.75,0);
\draw[thick,-|] (0,0) -- (0.5,0);
\draw[thick,-|] (0,0) -- (-0.5,0);
\draw[thick,-|] (0,0) -- (0.25,0.433);
\draw[thick,-|] (0.5,0.866) -- (0.75,0.433);
\draw[thick,-|] (0.5,0.866) -- (0.5,0.866+0.5);
\draw[thick,-|] (0,0) -- (-.65/3*2,-0.75/3);

\fill (0,0) circle (1.25pt);
\fill (1,0) circle (1.25pt);
\fill (0.5,0.866) circle (1.25pt);

\draw[<-] (0,0.15) arc (90:-220:0.15);
\node at (0,-0.25) {$\sigma_0$};
\draw[<-] (0.5,0.44) arc (90:-220:0.15);
\node at (0.5,0.287) {$\sigma_\infty$};
\draw[<->] (0.65,0.65) arc (90:-30:0.25);
\node at (1.05,0.55) {$\sigma_1$};
\end{tikzpicture} \end{center}
\caption[Combinatorial data]{An example of the permutations that define a ribbon graph around a boundary cycle.} \label{permutations}
\end{figure}

To be more precise, a ribbon graph  $\Gamma$ can be build out of combinatorial data in the following way. 
Let $H$ be a finite set of even cardinality together with a subset (possibly empty) $S \subset H$ such that $|S|$ is also even. 
Let $\sigma_1 \in \FS_H$ be an involution without fixed points under which $S$ is invariant and $\sigma_0 \in \FS_{H-S}$ be such that $\sigma_0$ is a product of cyclic permutations with disjoint support. A vertex of $\Gamma$ is then given as an orbit of $\sigma_0$, while an edge is an orbit of $\sigma_1$. 
The set of vertices can be identified with 
$V(\Gamma)=(H-S)/\sigma_0$ 
and the set of edges with $E(\Gamma)=H/\sigma_1$. 
Semistable circles correspond with pairs of half-edges in the orbit of $\sigma_1$ that belong to $S$.

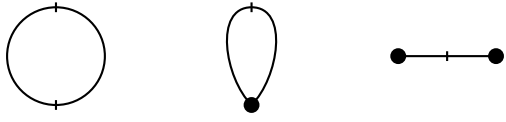
\begin{figure} 
\begin{center}
\begin{tikzpicture}[scale=1.3]

\draw[thick] (0,0) circle (0.5);
\draw[thick] (0,-0.55) -- (0,-0.45);
\draw[thick] (0,0.55) -- (0,0.45);

\draw[thick,xshift=2cm,yshift=-0.5cm,rotate=270] (0,0) to [out=135, in=90] (-1,0) to [out=270,in=225] (0,0);
\draw[thick,xshift=2cm] (0,0.45) -- (0,0.55);
\draw[fill=black,thick,xshift=2cm] (0,-0.5) circle (2pt);

\draw[thick,xshift=3.5cm] (0,0) -- (1,0);
\draw[thick,xshift=3.5cm] (0.5,0.05) -- (0.5,-0.05);
\draw[fill=black,thick,xshift=3.5cm] (0,0) circle (2pt);
\draw[fill=black,thick,xshift=4.5cm] (0,0) circle (2pt);

\end{tikzpicture}
\end{center}
\caption[Graphs with one edge]
{All three graphs with only one edge. 
The first is a semistable circle with two half-edges and no vertices. 
The second is a loop with one vertex. 
The third is an edge with two vertices. 
The first two graphs have two cusps while the third one has only one.} \label{oneedge}
\end{figure}

Let $\sigma_\infty = \sigma_0^{-1} \sigma_1 \in \FS_{H-S}$ where $\sigma_1$ is restricted to $H-S$. 
The set of {\bf cusps} is defined as $C(\Gamma)= (H-S)/\sigma_\infty \sqcup S$.
Note that in this way every semistable circle has two cusps represented by each half-edge. 
The half-edges in the orbit of a cusp form a cyclically ordered set of half-edges called a {\bf boundary cycle}. In the case of semistable circles the boundary cycle contains only one element (the corresponding half-edge associated with the cusp). Given a boundary cycle not coming from a cusp represented by a half-edge in $S$, its half-edges form a graph called a {\bf boundary subgraph}. 
For example, in Figure~\ref{permutations} the boundary cycle represented in the middle forms a boundary subgraph with three edges and three vertices (the middle triangle as a subgraph). The reason for such names will become evident later. 
For a semistable circle we let $\sigma_\infty$ be the identity. 

The cusps and the vertices of valence one or two will be called {\bf distinguished points}. 
Notice also that knowing $\sigma_1$ and $\sigma_\infty$ completely determines 
the ribbon graph structure since we have $\sigma_0= \sigma_1 \sigma_\infty^{-1}$.

A {\bf loop} is an edge incident to only one vertex. 
A {\bf tree} is a connected graph $T$ with trivial $\bar{H}_*(T)$. 

An {\bf isomorphism} of ribbon graphs is a graph isomorphism preserving the cyclic orders on each vertex. 
Therefore, two graphs $\Gamma$, $\Gamma'$ are isomorphic when 
there is a bijection $\eta: H \to H'$ between the set of half-edges of these two graphs 
that commutes with $\sigma_0$, $\sigma_0'$ and $\sigma_1$, $\sigma_1'$. 
In particular, the boundary cycles are preserved, 
\emph{i.e.}, $\eta$ also commutes with $\sigma_\infty$, $\sigma_\infty'$.  
If we restrict to automorphisms of a 
fixed graph it is clear that this definition will generate a group. The group of automorphisms of the semistable circle is $\Z/ 2 \Z$.

An {\bf orientation of an edge} can be defined as an order on its corresponding half edges 
and we can use the notation $\vec{e} = \overrightarrow{h \sigma_1(h)}$ where $h$ is a half-edge. 
The involution $\sigma_1$ switches the orientation of an edge.

\begin{figure} 
\begin{center}
\begin{tikzpicture}[scale=0.9]

\draw[thick] (0,0) circle (0.5);
\draw[thick] (-0.5,0) -- (0.5,0);
\draw[fill=black] (-0.5,0) circle (2pt);
\draw[fill=black] (0.5,0) circle (2pt);

\draw[thick,xshift=2cm] (0,0) to [out=135, in=90] (-1,0) to [out=270,in=225] (0,0);
\draw[thick,xshift=2.5cm,rotate=180] (0,0) to [out=135, in=90] (-1,0) to [out=270,in=225] (0,0);
\draw[thick,xshift=2cm] (0,0) -- (0.5,0);
\draw[fill=black,thick,xshift=2cm] (0,0) circle (2pt);
\draw[fill=black,thick,xshift=2cm] (0.5,0) circle (2pt);

\draw[thick,xshift=5cm] (0,0) to [out=135, in=90] (-1,0) to [out=270,in=225] (0,0);
\draw[thick,xshift=5cm,rotate=180] (0,0) to [out=135, in=90] (-1,0) to [out=270,in=225] (0,0);
\draw[fill=black,thick,xshift=5cm] (0,0) circle (2pt);

\draw[thick,xshift=7.5cm] (0,0) to [out=135, in=90] (-1,0) to [out=270,in=225] (0,0);
\draw[thick,xshift=7.5cm] (0,0) -- (0.7,0);
\draw[fill=black,thick,xshift=7.5cm] (0,0) circle (2pt);
\draw[fill=black,thick,xshift=7.5cm] (0.7,0) circle (2pt);

\draw[thick,xshift=9cm,yshift=-0.5cm,rotate=270] (0,0) to [out=135, in=90] (-1,0) to [out=270,in=225] (0,0);
\draw[fill=black,thick,xshift=9cm] (0,-0.5) circle (2pt);

\draw[thick,xshift=10cm] (0,0) -- (1,0);
\draw[fill=black,thick,xshift=10cm] (0,0) circle (2pt);
\draw[fill=black,thick,xshift=11cm] (0,0) circle (2pt);

\end{tikzpicture}
\end{center}
\caption[All (0,3) ribbon graphs]{Each connected ribbon graph shown generates a sphere with three labeled points.} \label{simplegraphs}
\end{figure}
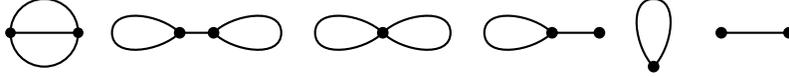

To every ribbon graph $\Gamma$ we associate an oriented surface $\surf(\Gamma)$ as follows. 
To each oriented edge $\vec{e} = \overrightarrow{h \sigma_1(h)}$ we attach at the base an oriented triangle $K_{\vec{e}}=C |\vec{e}|$, 
where $|\vec{e}|$ is homeomorphic to the closed unit interval and $C$ refers to the cone construction. 
The base of $K_{\vec{e}}$ is identified with the base of $K_{\sigma_1(\vec{e})}$. 
Next we paste the right-hand edge of $K_{\vec{e}}$ with the left-hand edge of $K_{\sigma_\infty(\vec{e})}$. 
All vertices opposite to the bases of all triangles in the same boundary cycle are identified. 
Such points are in one to one correspondence with the orbits of $\sigma_\infty$; 
that is why we call them cusps. Gluing triangles is done in a way compatible with the orientation of each triangle, 
and thus, this compact surface comes already triangularized. 
For each connected component $\Gamma_i$ of the graph, 
set $v_i=|V(\Gamma_i)|, e_i=|E(\Gamma_i)|$ and $n_i$ for the number of vertices, edges and cusps within, respectively.
Then an easy exercise shows that there are $2e_i$ faces, $3e_i$ edges and $v_i+n_i$ points. 
So, the corresponding surfaces $\surf(\Gamma_i)$ must be of genus $g_i=(2-v_i+e_i-n_i)/2$ 
when we think of the cusps as points that are removed from the surface.
This surfaces come with a natural orientation given by the tiles 
since they are naturally oriented and their orientations match each other because of the way we have glued them.

This construction can also be applied to semistable circles. Even though semistable circles have no vertices they still have half-edges and thus they also have two orientations corresponding to their boundary cycles. We glue the sides of each semi-infinite rectangle in order to obtain an infinite cylinder with two cusps. 
One may worry that since there is no vertex around, 
there is no precise way to know where to start gluing the rectangle. 
However, the choice of a base point becomes irrelevant because the moduli of semistable spheres is trivial.

A {\bf $P$-labeled ribbon graph} is a ribbon graph together with 
an injection $x: P \hookrightarrow V(\Gamma) \sqcup C(\Gamma)$ whose image contains all distinguished points, 
that is, cusps and vertices of valence one or two.  
The elements of the image will be called {\bf labeled points}. 

An {\bf isomorphism} of $P$-labeled ribbon graphs is a ribbon graph isomorphism that preserves the labels. 
In particular, the automorphism group of the semistable circle is trivial.

The {\bf Euler characteristic} of a $P$-labeled ribbon graph is defined as 
the Euler characteristic of the graph minus the number of vertices of valence one or two. 
The semistable circle is defined to have Euler characteristic equal to zero. 
This is the same as the Euler characteristic of the compact surface $\surf(\Gamma)$ made of triangles 
but taking away the cusps and points in the surface corresponding to vertices of valence one and two. 
The reason to do this is because those points will correspond with labeled points after collapsing subgraphs (which will be defined later). 

\begin{rem}
Clearly, if $\Gamma$ is a $P$-labeled ribbon graph then 
$\surf(\Gamma)$ inherits a $P$-labeling in the form of a function $x: P \hookrightarrow \surf(\Gamma)$. 
The topological type $(g,|P|)$ of a $P$-labeled ribbon graph refers to 
the genus $g$ of the generated surface and the number $|P|$ of labels. 
It is also easy to check that the Euler characteristic of the ribbon graph is 
the same as the Euler characteristic of the surface minus the cusps and points corresponding with vertices of valence one and two. 
Figure~\ref{simplegraphs} shows all ribbon graphs that generate a sphere with three labeled points. 
The two on the left are the only three-valent graphs, 
the first one is called 
{\bf theta graph} ($\Theta$) 
and let us call the second one {\bf omega graph} ($\Omega$).
\end{rem}

Fix a vertex $v$ in a ribbon graph. 
We can construct a new ribbon graph by replacing $v$ with a regular $|H_v|$-gon whose vertices are
attached to the elements of $H_v$ in the obvious manner. 
This can be done in a way in which we also induce 
cyclic orders on the new vertices created as in Figure~\ref{blowup}. 
The new ribbon graph is the {\bf blow-up} of $v$. 
This operation adds one extra boundary cycle to the ribbon graph. 
In the case of a vertex of valence one the blow-up adds a loop based at such vertex. 
In the case of valence two the blow-up replaces the vertex with two separate vertices 
and adds two new edges joining those edges forming a new boundary cycle where the old vertex was.

\begin{figure}
\begin{center}
\begin{tikzpicture}[scale=1.2]
\draw[fill=black] (0,0) circle (2pt);
\draw[thick] (1,0) -- (0,0);
\draw[thick] (0,0) -- (-0.5,0.866);
\draw[thick] (0,0) -- (-0.5,-0.866);
\end{tikzpicture}
\hspace{2cm}
\begin{tikzpicture}[scale=0.6]
\draw[fill=black] (1,0) circle (2.5pt);
\draw[fill=black] (-0.5,0.866) circle (2.5pt);
\draw[fill=black] (-0.5,-0.866) circle (2.5pt);

\draw[thick] (1,0) -- (-0.5,0.866);
\draw[thick] (-0.5,0.866) -- (-0.5,-0.866);
\draw[thick] (-0.5,-0.866) -- (1,0);
\draw[thick] (1,0) -- (2,0);
\draw[thick] (-0.5,0.866) -- (-1,1.733);
\draw[thick] (-0.5,-0.866) -- (-1,-1.733);
\end{tikzpicture}
\end{center}

\caption[Blowup]{The blow-up of a three-valent vertex.}  \label{blowup}
\end{figure}
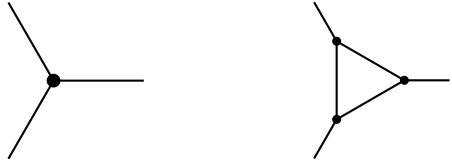

A boundary cycle is called {\bf injective} 
if any two half-edges in this orbit are not in the same orbit of $\sigma_0$ or $\sigma_1$. 
This implies that the boundary subgraph is homeomorphic to a circle. For example, the extra boundary cycle generated in the blow-up is always injective. A non-injective boundary cycle occurs for example this boundary cycle goes around a part of the graph that raises the genus of the graph.

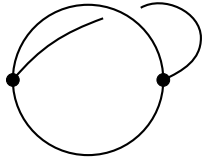
\begin{figure}

\begin{center}
\begin{tikzpicture}[scale=2]

\draw[thick] (0,0) circle (0.5);
\draw[thick] (-0.5,0) to [out=50, in=200] (0.1,0.41);
\draw[thick] (0.35,0.48) to [out=30 , in=90] (0.75,0.28) to [out=270, in=25] (0.5,0);
\draw[fill=black] (-0.5,0) circle (1.2pt);
\draw[fill=black] (0.5,0) circle (1.2pt);

\end{tikzpicture}
\end{center}

\caption[Torus graph]{The graph shown generates a torus with one label point. All half-edges of the graph form a boundary cycle and hence it is not injective.}  \label{blowup}
\end{figure}

By {\bf disjoint boundary cycles} we mean boundary cycles that do not share any 
half edges in the same orbit of $\sigma_0$ or $\sigma_1$. 
For them the associated boundary subgraphs do not intersect.

Given two disjoint boundary cycles with at least one of them being injective, 
we can produce a finite family of ribbon graphs as follows.

Since both boundary cycles correspond with subgraphs that can be identified with CW-complexes themselves, 
choose parametrizations of each subgraph by $\S^1$. 
The parametrization of the subgraph associated to the injective boundary cycle 
must be compatible with the natural counter-clockwise orientation of $\S^1 \subset \C$, \emph{i.e.}, it must follows the cyclic order of the boundary cycle. The other subgraph is parametrized with the opposite orientation. We glue both subgraphs by identifying points with the same parametrization. This gives a new graph called a {\bf gluing}. 
This resulting graph have as set of vertices the union of the vertices of the two previous graphs (where some of them might be identified due to coinciding 
in the parametrization) with possibly new edges. Those new edges are generated 
when a vertex of one of the subgraphs is glued to the interior of an edge of the other edge, thus splitting the edge. 
It is worthwhile to note that the euler characteristic is additive under this gluing operation because topologically this operation is the same as the connected sum applied to the surfaces associated to the corresponding graphs. 
For the ribbon graph structure, the cyclic order on each vertex induces a cyclic order on each vertex of the gluing in a natural way 
(in case two vertices are glued together, one from each subgraph, 
a new cyclic order on the resulting vertex is created by the concatenation of the original cyclic orders: this can be done in a well defined manner due to the parametrizations of the boundary cycles). The new edges thus created can be divided in half-edges and the parametrizations used so far can be forgetten altogether if we keep track of the combinatorial cyclic data. Hence, this process results in a new ribbon graph depending on the parametrization thus chosen.
(This is why disjointness and injectivity of the boundary cycles was introduced). 
An example is presented in Figure~\ref{gluings}. 
There is also a way to define this gluing construction in a purely combinatorial way 
but it lacks the geometrical intuition given by the present one.

In order to produce a family of ribbon graphs we can change the parametrizations and keep only one representative from each isomorphism class of ribbon graph thus created. Since there is a bound on the size of the resulting graphs, and the possible combinatorics for each size are also finite, there will be only a finite number of isomorphism classes. 

\begin{df}
Given a vertex and a boundary cycle whose associated graph does not include the given vertex, we define a {\bf gluing} by applying the gluing construction to 
the blow-up of the vertex and the given boundary cycle. \label{gluing}
\end{df}

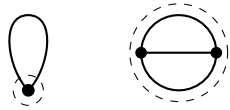
\begin{figure} 

\begin{center}
\begin{tikzpicture}

\draw[thick,yshift=-0.5cm,rotate=270] (0,0) to [out=135, in=90] (-1,0) to [out=270,in=225] (0,0);
\draw[fill=black,thick] (0,-0.5) circle (2pt);
\draw[dashed] (0,-.5) circle (0.2cm);

\draw[thick,xshift=2cm] (0,0) circle (0.5);
\draw[thick,xshift=2cm] (-0.5,0) -- (0.5,0);
\draw[fill=black,xshift=2cm] (-0.5,0) circle (2pt);
\draw[fill=black,xshift=2cm] (0.5,0) circle (2pt);
\draw[dashed, xshift=2cm] (0,0) circle (0.65cm);

\end{tikzpicture}
\end{center}

\caption[twographs]
{Two disjoint ribbon graphs before the gluing process. 
The vertex to be blown-up and the boundary cycle that will be glued are highlighted by the dashed circles.} \label{twographs}
\end{figure}

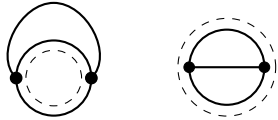
\begin{figure}
\begin{center}
\begin{tikzpicture}

\draw[thick] (-0.5,0) to [out=135, in=180] (0,1) to [out=0,in=45] (0.5,0);
\draw[fill=black,thick] (-0.5,0) circle (2pt);
\draw[fill=black,thick] (0.5,0) circle (2pt);

\draw[thick] (0,0) circle (0.5);
\draw[fill=black] (-0.5,0) circle (2pt);
\draw[fill=black] (0.5,0) circle (2pt);
\draw[dashed] (0,0) circle (0.37cm);

\end{tikzpicture}
\qquad
\begin{tikzpicture}

\draw[thick,xshift=2cm] (0,0) circle (0.5);
\draw[thick,xshift=2cm] (-0.5,0) -- (0.5,0);
\draw[fill=black,xshift=2cm] (-0.5,0) circle (2pt);
\draw[fill=black,xshift=2cm] (0.5,0) circle (2pt);
\draw[dashed, xshift=2cm] (0,0) circle (0.65cm);

\end{tikzpicture}

\end{center}

\caption[blowupdisjoint]
{After the blow-up of the highlighted vertex 
in Figure~\ref{twographs}, left, we obtain two disjoint boundary cycles as highlighted 
in the figure.} \label{blowupdisjoint}
\end{figure}

\begin{figure}
\begin{center}
\begin{tikzpicture}

\draw[thick] (-0.5,0) to [out=135, in=180] (0,1) to [out=0,in=45] (0.5,0);
\draw[fill=black,thick] (-0.5,0) circle (2pt);
\draw[fill=black,thick] (0.5,0) circle (2pt);

\draw[thick] (0,0) circle (0.5);
\draw[fill=black] (-0.5,0) circle (2pt);
\draw[fill=black] (0.5,0) circle (2pt);

\draw[fill=black] (-0.3546,-0.3546) circle (2pt);
\draw[fill=black] (0.3546,-0.3546) circle (2pt);

\draw[thick] (-0.3546,-0.3546) to [out=45, in=135] (0.3546,-0.3546);

\end{tikzpicture}
\qquad
\begin{tikzpicture}

\draw[thick] (-0.5,0) to [out=135, in=180] (0,1) to [out=0,in=45] (0.5,0);
\draw[fill=black,thick] (-0.5,0) circle (2pt);
\draw[fill=black,thick] (0.5,0) circle (2pt);

\draw[thick] (0,0) circle (0.5);
\draw[thick] (0,-0.5) -- (0,0.5);
\draw[fill=black] (-0.5,0) circle (2pt);
\draw[fill=black] (0.5,0) circle (2pt);

\draw[fill=black,thick] (0,-0.5) circle (2pt);
\draw[fill=black,thick] (0,0.5) circle (2pt);

\end{tikzpicture}
\qquad
\begin{tikzpicture}

\draw[thick] (-0.5,0) to [out=135, in=180] (0,1) to [out=0,in=45] (0.5,0);
\draw[fill=black,thick] (-0.5,0) circle (2pt);
\draw[fill=black,thick] (0.5,0) circle (2pt);

\draw[thick] (0,0) circle (0.5);
\draw[fill=black] (-0.5,0) circle (2pt);
\draw[fill=black] (0.5,0) circle (2pt);

\draw[fill=black] (-0.3546,0.3546) circle (2pt);
\draw[fill=black] (0.3546,0.3546) circle (2pt);

\draw[thick] (-0.3546,0.3546) to [out=315, in=225] (0.3546,0.3546);

\end{tikzpicture}
\end{center}
\caption[Gluings]
{A few members of the family of ribbon graphs resulting from the gluing process of the previous 
two graphs. 
Here are shown only the trivalent members of that family.} \label{gluings}
\end{figure}
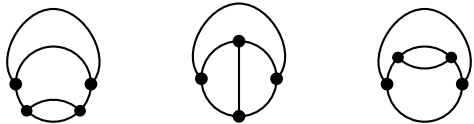

This construction is well defined because the blow-up is injective and 
the condition of the vertex being disjoint from the boundary cycle 
implies that the boundary cycles are disjoint. Compare to Figure~\ref{twographs}. 
\bigskip

\subsection{Semistable Ribbon Graphs}

Let us describe two ribbon graphs we can obtain from a proper subset of edges $Z \subset E(\Gamma)$. 
One will be associated to $Z$ and the other to its complement in $E(\Gamma)$. 
Denote by $\Gamma_Z$ the subgraph with set of edges $Z$ and $H_Z$ its set of half-edges. 
The ribbon graph structure is induced by $\sigma_0$ and $\sigma_1$ in the following way. 
The new $\sigma^{\Gamma_Z}_1$ is just the restriction, 
while $\sigma^{\Gamma_Z}_0$ is defined by declaring $\sigma^{\Gamma_Z}_0(h)$, with $h \in H_Z$, 
to be the first term in the sequence $(\sigma_0^k(h))_{k>0}$ that is in $H_Z$.

The proper subset $Z \subset E(\Gamma)$ of edges of a ribbon graph 
also induces a ribbon graph structure on the graph determined by the complement of $Z$ in $E(\Gamma)$. 
We will denote this graph by $\Gamma / \Gamma_Z$. 
The new graph has set of edges $E(\Gamma) - Z$ with induced set of half-edges $H_{\Gamma/\Gamma_Z}$. 
Since $\sigma_1$ and $\sigma_\infty$ completely determine the ribbon graph structure, 
it is enough to define them in $H_{\Gamma/\Gamma_Z}$. 
The new involution is just the restriction $\sigma^{\Gamma/\Gamma_Z}_1= \sigma_1 |_{H_{\Gamma/\Gamma_Z}}$. 
Given $h \in H_{\Gamma/\Gamma_Z}$ we define $\sigma^{\Gamma/\Gamma_Z}_\infty (h)$ to be 
the first term of the sequence $(\sigma_\infty^k(h))_{k>0}$ in $H_{\Gamma/\Gamma_Z}$.

A few remarks about the nature of $\Gamma / \Gamma_Z$ are in order. 
For instance, if $Z$ is just one edge with two different vertices (and hence not a loop) then $\Gamma / \Gamma_Z$ 
is topologically the result of collapsing that edge to a point as in Figure~\ref{collapseedge}, 
where two vertices are amalgamated into one creating a new cyclic order around that vertex. 
The same is true if one replaces a single edge with any set of edges $Z$ such that $\Gamma_Z$ is simply connected. 
This can be shown by induction on the number of elements of $Z$. 

\begin{figure} 
\begin{center}
\begin{tikzpicture}
\draw[thick] (0,0) to [out=135, in=90] (-1,0) to [out=270,in=225] (0,0);
\draw[thick,xshift=0.5cm,rotate=180] (0,0) to [out=135, in=90] (-1,0) to [out=270,in=225] (0,0);
\draw[thick] (0,0) -- (0.5,0);
\draw[fill=black,thick] (0,0) circle (2pt);
\draw[fill=black,thick] (0.5,0) circle (2pt);

\draw[thick,xshift=5cm] (0,0) to [out=135, in=90] (-1,0) to [out=270,in=225] (0,0);
\draw[thick,xshift=5cm,rotate=180] (0,0) to [out=135, in=90] (-1,0) to [out=270,in=225] (0,0);
\draw[fill=black,thick,xshift=5cm] (0,0) circle (2pt);

\node at (0.3,0.5) {$e$};
\node at (0.3,-0.9) {$\Gamma$};
\node at (5.2,-0.9) {$\Gamma/\Gamma_{\{ e \}}$};
\end{tikzpicture}
\end{center}
\caption[Collapsing an edge]
{Collapsing an edge that is not a loop.} \label{collapseedge}
\end{figure}
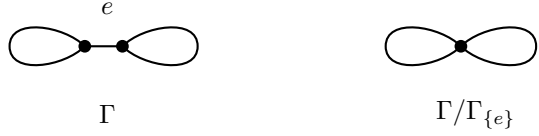

However, in general, $\Gamma / \Gamma_Z$ is not a topological quotient of $\Gamma$ over $Z$. 
Figure~\ref{nottopquo} shows an example of this last case. 

It turns out that this definition allows us to track the creation of nodes at the graph level. 
In this last example the collapse of a loop creates a new vertex and the cyclic order around the only vertex in $\Gamma$ splits into two.

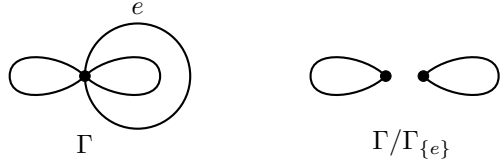
\begin{figure} 
\begin{center}
\begin{tikzpicture}
\draw[thick] (0,0) to [out=135, in=90] (-1,0) to [out=270,in=225] (0,0);
\draw[thick, rotate=180] (0,0) to [out=135, in=90] (-1,0) to [out=270,in=225] (0,0);
\draw[thick] (0.7,0) circle (0.7);
\draw[fill=black] (0,0) circle (2pt);
\node at (0.7,0.9) {$e$};
\node at (0,-0.9) {$\Gamma$};

\draw[thick, xshift=4cm] (0,0) to [out=135, in=90] (-1,0) to [out=270,in=225] (0,0);
\draw[fill=black, xshift=4cm] (0,0) circle (2pt);
\draw[thick, rotate=180, xshift=-4.5cm] (0,0) to [out=135, in=90] (-1,0) to [out=270,in=225] (0,0);
\draw[fill=black, xshift=4.5cm] (0,0) circle (2pt);
\node at (4.35,-0.9) {$\Gamma/\Gamma_{\{ e \}}$};
\end{tikzpicture}
\end{center}
\caption[Not a topological quotient]
{The original graph has one vertex while the second one has two and is disconnected.} \label{nottopquo}
\end{figure}

We now describe how to collapse edges in a $P$-labeled ribbon graph 
without changing the homeomorphism type of $\surf(\Gamma)$ relative to $P$.

A subset $Z \subset E(\Gamma)$ of a $P$-labeled ribbon graph $\Gamma$ is called {\bf negligible} 
if each connected component of $\Gamma_Z$ is either 
a tree with at most one labeled point or 
a homotopy circle without labeled points that contains a boundary subgraph.

\begin{df}
If $\Gamma$ is a $P$-labeled ribbon graph and $Z \subset E(\Gamma)$ is a negligible subset 
define the {\bf edge collapse} of $\Gamma$ respect to $Z$ as 
$\Gamma/Z = \Gamma/\Gamma_Z$ with the induced $P$-labeling. \label{edgecollap}
\end{df}

\begin{rem}
Collapsing a tree with at most one labeled point does not change the injectivity of the labels. 
Collapsing a subgraph that is homotopic to a circle without labeled points and contains a boundary subgraph 
is called a {\bf total collapse} 
and in this case the label of the cusp corresponding with the boundary subgraph 
switches into a label of the induced vertex after that collapse. 
The injectivity of this labeling is still 
preserved.
\end{rem}

\begin{lem}
If $Z$ is negligible then $\surf(\Gamma) \cong  \surf(\Gamma/Z)$ relative to $P$.
\end{lem}
\begin{proof}
It is possible to exhibit a sequence of homeomorphisms starting at $\surf(\Gamma)$ and ending at $\surf(\Gamma/Z)$. 
If a connected component of $\Gamma_Z$ is a tree with at most one labeled point let $e$ be an edge in said tree. 
As $e$ is contracted, the result on the associated surface is to contract 
the triangle $K_{\vec{e}}$ to an interval 
(one vertex goes to one vertex of the interval and 
the edge opposite to this vertex is contracted to the other vertex of the interval). 
This can be done to all edges of the tree without changing the injectivity of the labels. 
The same can be done on a homotopy circle without labeled points that contains a boundary subgraph. 
The difference is that in the last step we have a loop being contracted to the vertex at its base 
which will now have a label. 
This collapse also respects the injectivity of the labels because 
the homotopy circle did not have a labeled point on it. 
Such process does not change the homeomorphism type of the surface.
\end{proof}

It is also possible to collapse more general graphs 
by allowing only mild degenerations. If more arbitrary subsets of edges are collapsed, the homeomorphism type is not preserved but it can shown that the singularities thus obtained resemble the well-behaved double-point singularities of stable surfaces.  We start with the following notion.

A proper set of edges $Z$ is {\bf semistable} 
if no component of $\Gamma_Z$ is the set of edges of 
a negligible subset and every univalent vertex of $\Gamma_Z$ is labeled.  

\begin{rem} \label{trees}
If $Z$ is semistable then 
every contractible component of $\Gamma_Z$ contains at least two labeled points (otherwise it would be negligible). 
A component that is a homotopy circle without labeled vertices is necessarily a topological circle 
because univalent vertices must be labeled. 
It is also not a boundary subgraph of $\Gamma$ or else it would be negligible. 
\end{rem}

\begin{lem}
Given a ribbon graph $\Gamma$ every proper subset $Z \subset E(\Gamma)$ 
contains a unique maximal semistable subset $Z^{sst} \subset Z$. \label{maxsemis}
\end{lem}

\begin{proof}
Starting from $Z$ remove all edges containing an unlabeled vertex of valence one. 
Repeat this process until we cannot delete any further edges. 
All remaining univalent vertices are labeled. 
Now throw away all boundary subgraphs with no labeled vertices. The order here is important: first prunning away edges and then removing boundary subgraphs. This is because $Z$ can have components that are the union of boundary subgraphs with edges containing an unlabeled vertex of valence one, hence doing it in the opposite order would not remove at the end the boundary subgraph.

\begin{figure}
\begin{center}

\begin{tikzpicture}[scale=0.6]
\draw[fill=black] (1,0) circle (2.5pt);
\draw[fill=black] (-0.5,0.866) circle (2.5pt);
\draw[fill=black] (-0.5,-0.866) circle (2.5pt);

\draw[thick] (1,0) -- (-0.5,0.866);
\draw[thick] (-0.5,0.866) -- (-0.5,-0.866);
\draw[thick] (-0.5,-0.866) -- (1,0);
\draw[thick] (1,0) -- (2,0);
\draw[thick] (-0.5,0.866) -- (-1,1.733);
\draw[thick] (-0.5,-0.866) -- (-1,-1.733);

\node at (3.5,0) {$\longrightarrow$};
\node at (9,0) {$\longrightarrow$};
\node at (11,0) {$\varnothing$};

\draw[xshift=6cm, fill=black] (1,0) circle (2.5pt);
\draw[xshift=6cm, fill=black] (-0.5,0.866) circle (2.5pt);
\draw[xshift=6cm, fill=black] (-0.5,-0.866) circle (2.5pt);

\draw[xshift=6cm, thick] (1,0) -- (-0.5,0.866);
\draw[xshift=6cm, thick] (-0.5,0.866) -- (-0.5,-0.866);
\draw[xshift=6cm, thick] (-0.5,-0.866) -- (1,0);
\end{tikzpicture}

\end{center}

\caption[Blowup]{Turning a set $Z$ into $Z^{sst}$. The graph on the left has no labeled points and so in this case $Z^{sst}=\varnothing$.}  \label{pruning}
\end{figure}
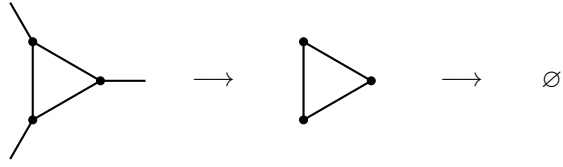

At the end what remains is $Z^{sst}$ and we have  $Z^{sst} = \varnothing$ if and only if $Z$ is negligible (see Figure~\ref{pruning}). 
The uniqueness of $Z^{sst}$ follows from the construction.
\end{proof}

Let $Z$ be a semistable subset of $\Gamma$. 
The {\bf reduction} of $\Gamma_Z$ is 
the result of deleting unlabeled vertices of valence two. 
This does not change the Euler characteristic since the vertices are unlabeled. We denote the reduction by $\hat{\Gamma}_Z$. 
The reduction of a homotopy circle with no labeled vertices corresponding with a semistable subset 
is in fact a semistable circle.

A semistable subset $Z$ is {\bf stable} if 
every component of $\Gamma_Z$ that is a topological circle contains a labeled vertex. 
An arbitrary proper subset $Z$ contains a unique maximal stable subset $Z^{st}$, 
which is obtained from $Z^{sst}$ by getting rid of the components that are topological circles without labeled vertices.
\medskip 

If $\Gamma$ is a $P$-labeled ribbon graph and $Z \subset E(\Gamma)$ is an arbitrary subset,  
define the {\bf edge collapse of $\Gamma$ respect to $Z$} 
as the disjoint union \[ \Gamma/Z = \Gamma/\Gamma_Z  \sqcup \hat{\Gamma}_{Z^{sst}} \] with the induced $P$-labeling.

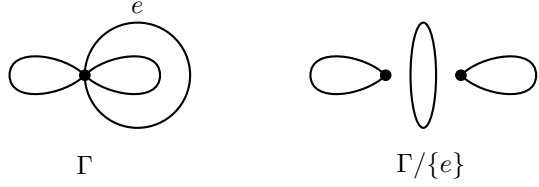
\begin{figure} 
\begin{center}
\begin{tikzpicture}
\draw[thick] (0,0) to [out=135, in=90] (-1,0) to [out=270,in=225] (0,0);
\draw[thick, rotate=180] (0,0) to [out=135, in=90] (-1,0) to [out=270,in=225] (0,0);
\draw[thick] (0.7,0) circle (0.7);
\draw[fill=black] (0,0) circle (2pt);
\node at (0.7,0.9) {$e$};
\node at (0,-1.2) {$\Gamma$};

\draw[thick, xshift=4cm] (0,0) to [out=135, in=90] (-1,0) to [out=270,in=225] (0,0);
\draw[fill=black, xshift=4cm] (0,0) circle (2pt);
\draw[thick, rotate=180, xshift=-5cm] (0,0) to [out=135, in=90] (-1,0) to [out=270,in=225] (0,0);
\draw[fill=black, xshift=5cm] (0,0) circle (2pt);
\node at (4.6,-1.2) {$\Gamma/{\{ e \}}$};
\draw[thick] (4.5,0) ellipse (0.17cm and 0.7cm);
\end{tikzpicture}
\end{center}
\caption[exampleofcollapse]
{Edge collapse of $\Gamma$ respect to the loop formed by $\{e\}$. 
On the right we have $\Gamma/\{ e\} = \Gamma/\Gamma_{\{e\}} \cup \hat{\Gamma}_{\{e\}^{sst}}$ 
where $\hat{\Gamma}_{\{e\}^{sst}}$ corresponds with the semistable circle.} \label{nottopquo}
\end{figure}

\begin{rem}
This generalizes Definition~\ref{edgecollap} because when $Z$ is negligible, 
we should have $Z^{sst} = \varnothing$ by Lemma~\ref{maxsemis}. 
Also notice that if $Z_1 \cap Z_2 = \varnothing$  
then $\Gamma / (Z_1 \sqcup Z_2) = (\Gamma/Z_1)/Z_2$. \label{splitcollapse}
\end{rem}

The next step is to introduce a generalization of ribbon graphs that will 
give a cellular decomposition of the decorated moduli space of semistable Riemann surfaces. 
This is similar to Looijenga's definition in \cite[Section 9.1]{loo}
but some changes were required. 

Let $Z$ be semistable. As $Z$ collapses this generates a new graph with possibly new vertices and new boundary cycles that will be call exceptional (see Figure~\ref{exceptionalv}). 
A vertex in $\Gamma/\Gamma_Z$ is represented by an orbit $V \in \sigma^{\Gamma/\Gamma_Z}_0$. 
If any of the elements in $V$ is the image under $\sigma_0$ of an element of $H_Z$ we call that vertex {\bf exceptional}. This means that some of the half-edges associated to this vertex belong to $H_Z$ and some do not.
In such case there is a corresponding orbit $B \in \sigma^{\Gamma_Z}_\infty$ that is not an orbit of $\sigma_\infty$ and has half-edges in $H_Z$. Such $B \in \sigma^{\Gamma_Z}_\infty$ is called an {\bf exceptional boundary cycle} and the associated subgraph an {\bf exceptional boundary subgraph}. 
Intuitively the exceptional vertex and exceptional boundary cycle correspond 
with the two points resulting from the normalization of a double node singularity in $\surf(\Gamma)$ after collapsing $Z$.

\begin{figure}
\begin{center}

\begin{tikzpicture}

\draw[thick] (-0.5,0) to [out=135, in=180] (0,1) to [out=0,in=45] (0.5,0);
\draw[fill=black,thick] (-0.5,0) circle (2pt);
\draw[fill=black,thick] (0.5,0) circle (2pt);

\draw[thick] (0,0) circle (0.5);
\draw[thick] (0,-0.5) -- (0,0.5);
\draw[fill=black] (-0.5,0) circle (2pt);
\draw[fill=black] (0.5,0) circle (2pt);

\draw[fill=black,thick] (0,-0.5) circle (2pt);
\draw[fill=black,thick] (0,0.5) circle (2pt);

\node at (1.8,0) {$\longrightarrow$};

\draw[xshift=3cm, thick,yshift=-0.5cm,rotate=270] (0,0) to [out=135, in=90] (-1,0) to [out=270,in=225] (0,0);
\draw[xshift=3cm, fill=black,thick] (0,-0.5) circle (2pt);
\draw[xshift=3cm, dashed] (0,-.5) circle (0.2cm);

\draw[thick,xshift=5cm] (0,0) circle (0.5);
\draw[thick,xshift=5cm] (0,-0.5) -- (0,0.5);
\draw[fill=black,xshift=5cm] (0,-0.5) circle (2pt);
\draw[fill=black,xshift=5cm] (0,0.5) circle (2pt);
\draw[dashed, xshift=5cm] (0,0) circle (0.65cm);

\node at (0,-1.2) {$\Gamma$};
\node at (3,-1.2) {$\Gamma/\Gamma_Z$};
\node at (4,-1.2) {$\sqcup$};
\node at (5.3,-1.2) {$\hat{\Gamma}_{Z^{sst}}$};

\end{tikzpicture}

\end{center}
\caption[Exceptional vertex and boundary component]
{On the left $Z$ represents the set of edges of the theta subgraph. On the right, after the collapse of $Z$, the exceptional vertex and exceptional boundary component are highlighted} \label{exceptionalv}
\end{figure}
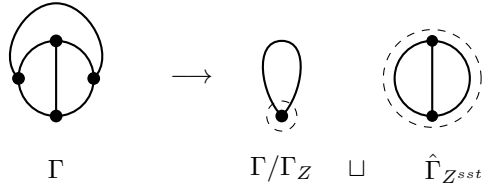

Consider an involution $\iota$ without fixed points on a subset $N \subset V(\Gamma) \sqcup C(\Gamma)$. 
The elements of $N$ will be called {\bf nodes}. 
Two elements of the same orbit are {\bf associated} and, in this case, 
we may also say that the corresponding connected components of the graph are associated. 
{\bf Cusp-nodes} and {\bf vertex-nodes} are defined in an obvious way. 
This involution allows us to identify points in $\surf(\Gamma)$. 
Denote by $\surf(\Gamma,\iota)$ the surface obtained from the previous identification. 
Let $\Gamma = \cup_{i \in {\mathcal I}} \Gamma_i$, where the $\Gamma_i$'s are the connected components of $\Gamma$. 
Thus we have $\pi_0(\Gamma) = \{ [\Gamma_i] \} \cong {\mathcal I}$. 
Set $V_i=V(\Gamma_i)$, $C_i=C(\Gamma_i)$ and $N'_i = N(\Gamma_i)$ for the vertices, cusps and nodes 
in the $i^{\text{th}}$ component of $\Gamma$. 

From now on we will only consider graphs with involutions for which the following properties apply. 

\begin{enumerate}
\item A connected component of the graph cannot be associated to itself.
\item Two semistable circles cannot be associated.
\item The two cusps of a semistable circle are nodes.
\item A cusp-node can only be associated to a vertex-node and vice versa.
\item The surface $\surf(\Gamma,\iota)$ must be connected.
\end{enumerate}

Let $\N=\{0,1,2,... \}$. 
An {\bf order} for $\Gamma$ is a function $\ord: \pi_0 (\Gamma) \to \N$ that satisfies the following properties.

\begin{enumerate}
\item[(i)] If $\ord([\Gamma_i])=k>0$ then there exist $j$ such that $\ord([\Gamma_j])=k-1$.
\item[(ii)] Let $p \in N'_i$ and $q=\iota(p) \in N'_j$. 
Then $p \in C_i$ if and only if  
$q \in V_j$ by property (4) on the previous list. 
In this case we require also that 
$\ord([\Gamma_j])<\ord([\Gamma_i])$
holds.
\end{enumerate}

The {\bf order} of an edge, half-edge, vertex or cusp is the order of the connected component where they sit. 

The next result gives some insight into the order function and is not hard to prove.

\begin{lem} \label{order}
Given an order $\ord: \pi_0(\Gamma) \to \N$ the following properties hold.

\begin{enumerate}
\item If $p\in N'_i$ and $\ord([\Gamma_i])=0$ then $p \in V$.
\item There is a constant $m \in \N$ such that $\ord([\Gamma_i])\le m$ for all $i$ and 
given $k$ such that $0 \le k \le m$ there exist $i$ with $\ord([\Gamma_i])=k$.\hfill$\square$
\end{enumerate}
\end{lem}

A {\bf semistable ribbon graph} is a ribbon graph $\Gamma$ together with an involution $\iota$ 
that satisfies conditions 1-5 as above and an order function $\ord$.

\begin{rem}
A ribbon graph can be viewed as a semistable ribbon graph with $N=\varnothing$. 
Notice also that we can obtain $\surf(\Gamma)$ from $\surf(\Gamma,\iota)$ by splitting identified points in two;  
this is a process commonly known as {\bf normalization}. 
When $N \ne \varnothing$ we call the graph {\bf singular}.
\end{rem}

A {\bf $P$-labeled semistable ribbon graph} is 
a semistable ribbon graph together with an inclusion $x:P \hookrightarrow V(\Gamma) \sqcup C(\Gamma)$ satisfying 
the following two conditions. 
\begin{enumerate}
\item The image $x(P)$ is disjoint from the set of nodes.
\item The union $x(P) \cup N$ contains all distinguished points, that is, all cusps and vertices of valence one or two. 
\end{enumerate}
This inclusion is called a {\bf $P$-labeling}. 
An {\bf isomorphism} in this case is an isomorphism of the underlying ribbon graph 
respecting the involution and order as well as the labeling. 
Nodes are preserved in this case because they are part of the data in the involution.

\begin{lem}
If $\Gamma$ is a $P$-labeled semistable ribbon graph, then every component of the normalization of $\surf(\Gamma, \iota)$ 
(\emph{i.e.}, a component of $\surf(\Gamma) - (N \sqcup x(P))$) has non-positive Euler characteristic.
\end{lem}
\begin{proof}
By definition we know that 
$\surf(\Gamma) - C(\Gamma)$ admits $\Gamma$ as a deformation retract. 
If a component is contractible then it must have at least two labeled points or nodes 
because such graph has at least two univalent vertices and the union $x(P) \cup N$ contains all distinguished points. 
This makes the Euler characteristic negative on those components. 
If the component is a topological circle, the Euler characteristic is zero if it is a semistable circle or negative if the circle has vertices correponding to marked points. 
In any other case the connected component of the graph has negative Euler characteristic.
\end{proof}

The surface $\surf(\Gamma, \iota)$ is called a {\bf $P$-labeled semistable surface} which inherits topological characteristics from the $P$-labeled semistable ribbon graphs that generates it.

We are almost ready to define the edge collapse for semistable ribbon graphs. 
The order function keeps track of how the graph degenerates, 
and to satisfy its definition we are not allowed to collapse 
all the edges associated to 
all components of a given order:  
otherwise we would be missing a number in the list of orders in contradiction with Lemma~\ref{order}. 
This is why we need the following concept. 
A subset of edges of a given $P$-labeled semistable ribbon graph is called {\bf collapsible} 
if it does not contain all edges of a fixed order. 

\begin{df}
A \textbf{negligible} subset $X$ of a $P$-labeled semistable ribbon graph $\Gamma$ is a negligible subset as defined before with the extra condition that 
if a connected component of $X$ is a homotopy circle without labeled points that contains a boundary subgraph, 
then said boundary subgraph should not correspond with a cusp-node.
\end{df}
 
If the extra condition in the previous definition were not satisfied, a negligible boundary subgraph corresponding to a cusp-node 
would collapse to an induced graph with an involution without fixed points 
that would associate a vertex-node with a newly generated vertex-node. 
This is in contradiction with the definition of semistable ribbon graph. Notice also that the extra condition only makes sense when 
we have a semistable ribbon graph structure (which carries an involution). 
Now, when doing a total collapse of a boundary subgraph corresponding to a cusp-node, this subgraph will not disappear.  
Instead, it will generate a semistable circle. 
The induced involution without fixed points will associate 
the old vertex-node and the newly generated vertex-node to both cusps of this semistable circle. 
In this way the induced involution satisfies the condition of only 
associating cusp-nodes with vertex-nodes and vice versa. 
Another consequence is that a semistable subset (as opposed to a negligible one) of a $P$-labeled semistable ribbon graph could possibly contain boundary subgraphs giving rise to cusp-nodes.

\begin{df}
If $\Gamma$ is a $P$-labeled semistable ribbon graph and $Z\subset E(\Gamma)$ is a collapsible subset of edges, 
the {\bf edge collapse} is a new $P$-labeled semistable ribbon graph defined as follows.
\begin{itemize}

\item As a $P$-labeled ribbon graph the edge collapse is $\Gamma/Z$. 
Notice that the change on the definition of negligible subset creates semistable circles 
for each total collapse of a homotopy circle without labeled points that corresponds to a cusp-node.

\item There is a new order function defined inductively. 
For this we express $Z$ as a disjoint union $Z = \sqcup Z_{i_k}$ so that $Z_{i_k}$ is the set of all edges of order $i_k$ in $Z$ where $k=\{1,2,3,...\}$. We can assume that the $i_k$'s are in ascending order.
Define the order function on $\Gamma / Z_{i_1}$ in the following way.
\begin{itemize}
\item All connected components in $\Gamma/\Gamma_{Z_{i_1}}$ of order less than or equal to ${i_1}$ keep their orders. 
\item The order of any connected component in $\hat{\Gamma}_{Z^{sst}_{i_1}}$ is ${i_1}+1$. 
\item All connected components in $\Gamma/\Gamma_{Z_{i_1}}$ of order greater than ${i_1}$ increase their order by one. 
\end{itemize}
This defines an order function on $\Gamma/Z_{i_1} = \Gamma/\Gamma_{Z_{i_1}} \cup \hat{\Gamma}_{Z^{sst}_{i_1}}$. 
By Remark~\ref{splitcollapse}, we can continue this process inductively on $k$ until we generate an order function for $\Gamma/Z$.

\item There are possibly new induced nodes together with an involution without fixed points. 
In the case of the total collapse of a homotopy circle without labeled points that corresponds to a cusp-node 
the old vertex-node and the newly generated vertex-node are 
associated to both cusp-nodes of the generated semistable circle. 
It can be showed following the inductive construction in the previous item 
that the resulting involution without fixed point satisfies the definition required by a semistable ribbon graph.
\end{itemize} \label{semiedgecollapse}
\end{df}

The previous definition is really a lemma which we state below.

\begin{lem} \label{semi}
The edge collapse of a collapsible subset of a $P$-labeled semistable ribbon graph 
produces a new $P$-labeled semistable ribbon graph of the same topological type 
but with possibly more components of higher order and more nodes.\hfill$\square$
\end{lem}

\begin{rem}
To obtain semistable ribbon graphs with higher orders 
we need to collapse several subsets of a ribbon graph consecutively. 
Therefore, the right notion of edge collapse in a category of $P$-labeled semistable ribbon graphs is 
that of consecutive collapse of collapsible subsets.
\end{rem}

\subsection{Permissible Sequences}

Fix a pair of associated nodes on a $P$-labeled semi-stable ribbon graph. 
A {\bf tangent direction} is a choice of gluing between 
the vertex-node and the boundary cycle corresponding to the cusp-node as in Definition~\ref{gluing}. 
This choice has to be compatible with the cyclic orders on the set of half-edges of the vertex-node 
and the edges of the graph associated to the exceptional boundary cycle corresponding to the cusp-node. 
We are just choosing then an element of the finite set of isomorphism classes of graphs created by the gluing construction.

A {\bf decoration by tangent directions} on a semistable ribbon graph is 
the choice of tangent directions for each pair of associated nodes.
An {\bf isomorphism} of semistable ribbon graphs decorated by tangent directions 
must preserve the tangent directions in the sense that 
there is an induced graph isomorphism on the corresponding gluings.

Given a $P$-labeled ribbon graph $\Gamma$, 
a {\bf permissible sequence} is a sequence \[ Z_\bullet = (E(\Gamma)=Z_0,Z_1,...,Z_k) \] 
such that $Z_i \subset Z_{i-1}^{sst}$, where the inclusion is strict. 
We call $k$ the length of the sequence. 
The pair $(\Gamma, Z_\bullet)$ denotes a labeled ribbon graph and a permissible sequence in it. 
If in addition all $Z_i$'s are semistable we call this a {\bf semistable sequence}. 
An {\bf isomorphism} of ribbon graphs with permissible sequences is a ribbon graph isomorphism that preserve the permissible sequences. 

\begin{rem}
The length of the sequence will correspond with the maximal order of an associated semistable ribbon graph. 
Notice also that there is a natural bijection between pairs of length zero and $P$-labeled ribbon graphs.
\end{rem}

A {\bf negligible} subset of $(\Gamma,Z_\bullet)$ is a 
sequence $D_\bullet = (D_0, D_1,...,D_k)$ such that all $D_i$ are negligible, 
and satisfy $D_i \subset D_{i-1}$ and $D_i \subset Z_i$. 
Call ${\mathcal N}(\Gamma,Z_\bullet)$ the set of negligible subsets of $(\Gamma,Z_\bullet)$. 

\begin{rem} \label{negli}
It is easy to check that we have a bijection between 
negligible subsets of $\Gamma$ and negligible subsets of $(\Gamma,Z_\bullet)$ by using the natural restriction. 
Moreover, we can collapse along negligible subsets in a similar way as we did before. 
Given a permissible sequence $Z_\bullet$ and negligible subset $D_\bullet$ 
we define the {\bf edge collapse} of $(\Gamma,Z_\bullet)$ along $D_\bullet$ as $(\Gamma / \Gamma_{D_0}, (Z/ D)_\bullet)$ 
where $(Z / D)_\bullet$ is the sequence induced by edge collapse. 
It can be shown that the result is also permissible and has the same length. 
\end{rem}

Now that we know how to collapse along negligible subsets, 
we also want to be able to collapse permissible sequences along semistable subsets 
but we need to be careful on how we define the new sequence. 
Let $(\Gamma,Z_\bullet)$ be a $P$-labeled ribbon graph together with a permissible sequence. 
A subset $S \subset E(\Gamma)$ is {\bf collapsible} with respect to $(\Gamma, Z_\bullet)$ if 
we have $Z_i \not\subset S$ for all $i$. 
This last definition is similar to the concept of collapsible subset for semistable ribbon graphs and serves the same function. On the next lemma we assume that $S$ is semistable in $\Gamma$, otherwise we can always collapse first $S-S^{sst}$ obtaining a permissible sequence of the same lenght as in the previous remark.

\begin{lem} \label{semistable}
Given a collapsible subset $S$ with respect to $(\Gamma,Z_\bullet)$ and semistable in $\Gamma$, we can induce a new permissible sequence $(Z / S)_\bullet$ inductively. 
\end{lem}
\begin{proof}
Let $i$ be the integer satisfying $S \subset Z_i$ and $S \subset \!\!\!\!\!\! / \,\,\, Z_{i+1}$. 
Then 
we have $(Z/S)_j = Z_j$ for $j \le i$. 
Set $(Z/S)_{i+1} = S \cup Z_{i+1}$ and $(Z/S)_{i+2} = Z_{i+1}$. 
Now, if $S\cap (Z_{i+1} - Z_{i+2}) \ne \varnothing$ then $(Z/S)_{i+3} = (S - Z_{i+1}^c ) \cup Z_{i+2}$ and $(Z/S)_{i+4} = Z_{i+2}$, 
otherwise 
we would get $(Z/S)_{i+3} = Z_{i+2}$. 
We can continue this process until the last step: 
either we exhaust all of $S$, 
meaning that the last element of the sequence will be 
$(Z/S)_l = Z_k$, 
or 
we get $(Z/S)_l = S - Z_k^c$ where $k$ is the length of $Z_\bullet$ and $l$ the length of the new sequence. 
The resulting sequence can be shown to be permissible and will have $l>k$ since $S$ is semistable in $\Gamma$. 
The resulting pair is then $(\Gamma, (Z/S)_\bullet)$. 
\end{proof}

\begin{prop}  \label{corres}
A $P$-labeled ribbon graph together with a permissible sequence $Z_\bullet$ 
can be used to construct a $P$-labeled semistable ribbon graph.
\end{prop}
\begin{proof}
For $i>0$ we can always collapse $Z_i - Z_i^{sst}$ since these sets are negligible due to maximality. 
Therefore we can assume that all $Z_i$ are semistable for $i>0$. 
The disjoint union $\Gamma / \Gamma_{Z_1} \sqcup \hat{\Gamma}_{Z_1}$ naturally inherits a semistable ribbon graph structure 
through the involution identifying exceptional vertices with their corresponding exceptional boundary cycles. 
The connected components of $\hat{\Gamma}_{Z_1 - Z_1^{st}}$ are semistable circles. 
The components in $\Gamma / \Gamma_{Z_1}$ only contain vertex-nodes and thus all those components have order zero. 
All the components of $\hat{\Gamma}_{Z_1}$ have at least one cusp-node associated to a vertex-node in a component of order zero and 
hence all those components have order one. 
The $P$-labeling naturally induces a $P$-labeling on the semistable ribbon graph. 
We can inductively apply this process to $\hat{\Gamma}_{Z_i}$ and $Z_{i+1}$, 
thus obtaining a $P$-labeled semistable ribbon graph 
$(\Gamma / \Gamma_{Z_1} \sqcup \hat{\Gamma}_{Z_1} / \Gamma_{Z_2} \sqcup \cdots \sqcup \hat{\Gamma}_{Z_k},\iota,x)$. 
\end{proof}

Now is turn to describe the connection between ribbon graphs with semistable sequences and 
semistable ribbon graphs with decorations by tangent directions. 

\begin{thm} \label{bijec}
There exists a natural bijection between 
isomorphism classes of $P$-labeled ribbon graphs with semistable sequences and 
isomorphism classes of $P$-labeled semistable ribbon graphs with decorations by tangent directions. 
This identification preserves isomorphism classes of negligible and collapsible semistable subsets (with respect to the given structures) 
and commutes with the edge collapse of the corresponding sets.
\end{thm}
\begin{proof}
Let $\Gamma$ be a $P$-labeled ribbon graph and $Z_\bullet$ a semistable sequence. 
This data generates a $P$-labeled semistable ribbon graph by Proposition~\ref{corres}. 
To obtain the decorations by tangent directions, 
it is enough to keep track of where the half-edges of a vertex-node were attached on the original graph. 
This correspondence naturally descends to a correspondence on isomorphism classes.

Now suppose we have a $P$-labeled semistable ribbon graph decorated by tangent directions. 
The decorations by tangent directions allow us to reconstruct 
a $P$-labeled ribbon graph 
by using the gluing construction on vertex-nodes and boundary cycles. 
Since this is defined only up to isomorphism, this correspondence is well defined on isomorphism classes. 
On a representative, every component of a semistable graph induces a subgraph of the ribbon graph. 
Together with the order, this defines a sequence of subgraphs $Z_\bullet$ in the ribbon graph up to isomorphism. 
It is not hard to check that this sequence will indeed be semistable.

These correspondences are inverses of each other on isomorphism classes by construction. 
Remark~\ref{negli} implies that negligible subsets are preserved and it also implies the commutativity with the edge collapse. 
For collapsible semistable subsets we also use the natural restriction and the gluing construction 
to track the image of these sets under the bijection. 
By the definitions, Lemma~\ref{semi} and Lemma~\ref{semistable}, 
collapsible semistable subsets are preserved by the bijection.
\end{proof}

\begin{rem}
In fact, 
it is possible to define a category of semistable ribbon graphs and another one of ribbon graphs with permissible sequences. 
After defining the right notion of morphism, the previous theorem can be extended to an equivalence of 
the appropriate categories.
\end{rem}

\section{The Semistable Ribbon Graph Complex}

When considering unital metrics on $P$-labeled semistable ribbon graphs with decorations by tangent directions 
giving rise to surfaces of genus $g$, 
one obtains $\MUC{g}{n}$ which is homeomorphic to $\MUD{g}{n}$ (\cite{zun}). 
We define a complex based on the natural orbicell decomposition of $\MUC{g}{n}$.

\subsection{Grading and Orientation}

Let $\Gamma$ represent a $P$-labeled semistable ribbon graph with decorations by tangent directions 
(from now on, simply a semistable ribbon graph). 
Let $m=\ord(\Gamma)$ be as in Lemma~\ref{order}
, part 2, and denote by $\Gamma_i$ the subgraph corresponding with order $i$. 
We will refer to $\Gamma_i$ as the {\bf piece} of order $i$ which may contain several connected components. Define the {\bf degree} of a semistable ribbon graph as \[ \deg \Gamma = \sum_{k = 0}^m \# E(\Gamma_k)-1. \]

Recall that if $V$ is an $n$-dimensional vector space, then $\det (V) = \bigwedge^n V$. 
Also, if $A$ is a set and $k$ a field, 
the vector space $k A$ is defined as the $|A|$-dimensional vector space generated by $A$.

An \textbf{orientation} of a semistable ribbon graph is a unit vector in \[ \det ( \R E(\Gamma) \oplus \R^m ) \] 
that we denote by ``$\ori$''. 
If $\{ e_i \} = E(\Gamma)$ and $\{o_i\}$ is the canonical basis of $\R^m$ so that $o_i$ represents the piece of order $i-1$ of $\Gamma$, an orientation can be represented as \[ \ori = [e_1, e_2, e_3, ..., o_1, o_2, ..., o_m]. \]
 
From the definition, an odd permutation changes the sign of the orientation while an even permutation does not.

Let \[ V^{g,n}_k = \Span_\R \{ (\Gamma, \ori) \}, \] 
where $\Gamma$ is a semistable ribbon graph of degree $k$ that gives rise to a surface of genus $g$ with $n$ labeled points. 
Consider the subspace 
\[ W^{g,n}_k = \Span_\R \{ (\Gamma, \ori) + (\Gamma, -\ori) \}, \] 
where $\Gamma$ is as before. 
Now define \[ G^{g,n}_k = V^{g,n}_k / W^{g,n}_k. \] 
In this graded vector space the relation $(\Gamma, -\ori) = -(\Gamma, \ori)$ is satisfied 
(it is understood that this refers to the equivalent classes in the quotient).

\subsection{Differential}

The operator $d_e : G^{g,n}_k \to G^{g,n}_{k-1}$ is defined by \[ d_e (\Gamma, \ori) = \sum_{e \in E(\Gamma)} (\Gamma/\{ e \}, \ori_e), \] 
where the sum is taken over all edges except for those with both ends being labeled points 
(which in the case of a loop refers to the common vertex). 
The orientation $\ori_e$ is induced by choosing a representative of the form $\ori = [e,...]$ and letting $\ori_e = [...]$, 
that is, deleting the factor corresponding to the edge ``$e$'' in the orientation after choosing a representative 
with such factor in front. 
Clearly, this operator does not change the topological type of the graph, 
and hence defines a map $G^{g,n}_* \to G^{g,n}_*$. 
It remains to show that this operator has degree -1:  This follows from the fact that contracting an edge reduces the number of edges by one in some piece $\Gamma_i$ without altering the others.

Define the operator $d_s : G^{g,n}_k \to G^{g,n}_{k-1}$ by \[ d_s (\Gamma, \ori) = \sum_{Z \subset \Gamma} (\Gamma/Z,\ori_Z),  \] 
where the sum is taken over all semistable subgraphs that are completely contained in a single piece of a fixed order. 
The induced orientation $\ori_Z$ is defined as follows. 
Let $[e_1, e_2, e_3, ..., o_1, o_2, ..., o_m]$ be a representative of $\ori$. 
If $Z \subset \Gamma_i$, 
then for $\ori_Z$ we choose the representative $[o_{i+1}',e_1, e_2, e_3, ..., o_1', o_2', ..., o_{m+1}']$ 
where $o_k' = o_k$ for $1 \le k \le i$ and $o'_k = o_{k-1}$ for $i+2 \le k \le m+1$, 
here $o_{i+1}'$ corresponds to the newly created piece of order $i+1$ and it is attached in front. 
As it was proved in Lemma~\ref{semi}, 
collapsing semistable subsets does not change the topological type; 
therefore this is a map 
from $G^{g,n}_* \to G^{g,n}_*$. 
To show that the operator has degree -1 consider $\Gamma/Z$. 
Since 
\[ E((\Gamma/Z)_k) = \begin{cases} E(\Gamma_k) & 1 \le k \le i-1 \\ E(\Gamma_i) - E(Z) & k = i \\ E(Z) & k = i+1 \\ E(\Gamma_{k-1}) & i+2 \le k \le m+1, \end{cases} \] 
we obtain 
\begin{align*} \deg (\Gamma/Z) & = \sum_{k=0}^{m+1} \# E((\Gamma/Z)_k)-1 \\
 & = \left( \sum_{k=0}^{i-1} \# E(\Gamma_k) - 1 \right) + (\# E(\Gamma_i) - \# E(Z) - 1) + \\
 & \qquad 
\hskip 0.5in + (\# E(Z) - 1)+ \left( \sum_{k=i+1}^{m} \# E(\Gamma_k) -1 \right) \\ 
& = \left( \sum_{k=0}^m \# E(\Gamma_k) -1 \right) - 1 = \deg(\Gamma) -1. 
\end{align*}

\begin{ex}
Consider the graph $\Gamma \in G^{0,5}_8$ in Figure~\ref{dt}. 
The set of edges is $E(\Gamma) = \{ a_1, a_2, a_3, a_4, b_1, b_2, b_3, b_4, c\}$. 
Let $A$ be the subgraph with set of edges $E(A) = \{ a_k \}$ and $B$ the graph with $E(B) = \{ b_k \}$. 
Collapsing the edges in $A$ yields the graph $\Gamma / A$ shown on the left of Figure~\ref{degdt}. 
Degenerating $A \cup B$ yields $\Gamma /(A \cup B)$, shown on the right of Figure~\ref{degdt}. 
The numbers on the right of the semistable graphs represent 
the order and 
the dotted lines represent how the pieces of different orders are joint. 
For the graph collapsing, 
the subgraph $B$ generates the same graph as collapsing the subgraph $A$ of the graph on the right. 
However, 
the orientations are not the same. 
Let \[ \ori = [a_1, a_2, a_3, a_4, b_1, b_2, b_3, b_4, c] \] 
be the orientation of the graph on the top of Figure~\ref{degdt}. 
The orientations of the collapsed graphs are 
\begin{align*}
\ori_A & = [o_1, a_1, a_2, a_3, a_4, b_1, b_2, b_3, b_4, c], \\
\ori_{A \cup B} & = [o_1, a_1, a_2, a_3, a_4, b_1, b_2, b_3, b_4, c]. 
\end{align*} 
In turn, collapsing $B$ in the graph of the left and collapsing $A$ in the graph of the right 
produces the same graph in the bottom but with orientations 
given by 
\begin{align*}
(\ori_A)_B & = [o_1',o_2', a_1, a_2, a_3, a_4, b_1, b_2, b_3, b_4, c], \\
(\ori_{A \cup B})_A & = [o_2',o_1', a_1, a_2, a_3, a_4, b_1, b_2, b_3, b_4, c].
\end{align*}
\end{ex}

\begin{figure}
\begin{center}
\begin{tikzpicture}[scale=1.1] 
\draw (0,0) circle (1cm);
\draw (0,1) -- (0,-1);
\draw (1,0) -- (3,0);
\draw (4,0) circle (1cm);
\draw (4,1) -- (4,-1);

\fill (0,1) circle (3pt);
\fill (0,-1) circle (3pt);
\fill (1,0) circle (3pt);
\fill (3,0) circle (3pt);
\fill (4,-1) circle (3pt);
\fill (4,1) circle (3pt);

\node at (-0.7,0) {$a_1$};
\node at (0.3,0) {$a_2$};
\node at (0.95,0.8) {$a_3$};
\node at (0.95,-0.8) {$a_4$};

\node at (4.7,0) {$b_1$};
\node at (3.7,0) {$b_2$};
\node at (3.05,0.8) {$b_3$};
\node at (3.05,-0.8) {$b_4$};

\node at (2,-0.3) {$c$};
\end{tikzpicture}
\end{center}
\caption{An element of $G^{0,5}_8$.} \label{dt}
\end{figure}
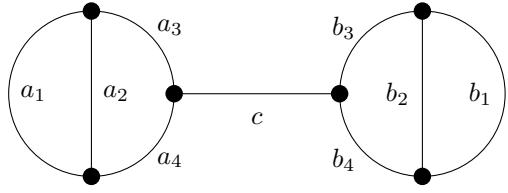

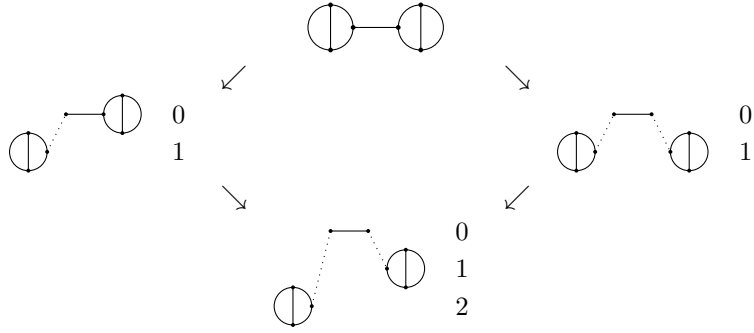
\begin{figure}
\[ \begin{array}{ccccc}
  &    & \begin{tikzpicture}[scale=0.3] 
\draw (0,0) circle (1cm);
\draw (0,1) -- (0,-1);
\draw (1,0) -- (3,0);
\draw (4,0) circle (1cm);
\draw (4,1) -- (4,-1);

\fill (0,1) circle (3pt);
\fill (0,-1) circle (3pt);
\fill (1,0) circle (3pt);
\fill (3,0) circle (3pt);
\fill (4,-1) circle (3pt);
\fill (4,1) circle (3pt);
\end{tikzpicture}  &   &   \\
  &  \swarrow  &   & \searrow  &   \\
\begin{tikzpicture}[scale=0.25]
\draw (0,0) circle (1cm);
\draw (0,1) -- (0,-1);

\draw[dotted] (1,0) -- (2,2);

\draw (2,2) -- (4,2);
\draw (5,2) circle (1cm);
\draw (5,3) -- (5,1);

\fill (0,1) circle (3pt);
\fill (0,-1) circle (3pt);
\fill (1,0) circle (3pt);

\fill (2,2) circle (3pt);
\fill (4,2) circle (3pt);
\fill (5,1) circle (3pt);
\fill (5,3) circle (3pt);

\node at (8,2) {0};
\node at (8,0) {1};
\end{tikzpicture}  &    &   &   & \begin{tikzpicture}[scale=0.25]
\draw (0,0) circle (1cm);
\draw (0,1) -- (0,-1);

\draw (2,2) -- (4,2);

\draw (6,0) circle (1cm);
\draw (6,1) -- (6,-1);

\draw[dotted] (1,0) -- (2,2);
\draw[dotted] (4,2) -- (5,0);

\node at (9,2) {0};
\node at (9,0) {1};

\fill (0,1) circle (3pt);
\fill (0,-1) circle (3pt);
\fill (1,0) circle (3pt);

\fill (2,2) circle (3pt);
\fill (4,2) circle (3pt);

\fill (5,0) circle (3pt);
\fill (6,-1) circle (3pt);
\fill (6,1) circle (3pt);
\end{tikzpicture}  \\
  &  \searrow  &   & \swarrow &   \\
  &    & \begin{tikzpicture}[scale=0.25]
\draw (0,0) circle (1cm);
\draw (0,1) -- (0,-1);

\draw (2,4) -- (4,4);

\draw (6,2) circle (1cm);
\draw (6,3) -- (6,1);

\draw[dotted] (1,0) -- (2,4);
\draw[dotted] (4,4) -- (5,2);

\node at (9,4) {0};
\node at (9,2) {1};
\node at (9,0) {2};

\fill (0,1) circle (3pt);
\fill (0,-1) circle (3pt);
\fill (1,0) circle (3pt);

\fill (2,4) circle (3pt);
\fill (4,4) circle (3pt);

\fill (5,2) circle (3pt);
\fill (6,1) circle (3pt);
\fill (6,3) circle (3pt);
\end{tikzpicture}
  &   &  
\end{array}
\]
\caption[Subgraph differential]{The ribbon graph on the top degenerates to different semistable ribbon graphs.}  \label{degdt}
\end{figure}

\begin{thm} \label{mainth}
The map $d : G^{g,n}_k \to G^{g,n}_{k-1}$ defined as $d = d_e + d_s$ is a differential.
\end{thm}

\begin{proof}
If it can be proved that $d_e^2=0$, $d_s^2=0$ and $d_e d_s + d_s d_e=0$ 
then the result follows since we have \[ (d_e + d_s)^2 = d_e^2 + d_e d_s + d_s d_e + d_s^2. \]

For $d_e$ consider two edges $a,b \in E(\Gamma)$ with the property that the set of vertices associated to these edges 
have at least two elements which are not labeled points and, further, if one of them is a labeled point 
then the associated edge is not a loop and the other vertex of this edge is not a labeled point. 
This condition guarantees that we can collapse both edges without modifying the Euler characteristic, 
and hence the topological type remains unchanged. 
Choose a representative of the orientation of $\Gamma$ of the form $[a,b, ...]$. 
Then, collapsing first $a$ and then $b$ induces the orientation $[...]$; 
but collapsing first $b$ and then $a$ changes the order and hence the sign, thus obtaining $-[...]$. 
This two different orders appear in $d_e^2$ and therefore we get $d_e^2 = 0$.

To show $d_s^2 = 0$ take two semistable subgraphs $W$ and $Z$ with 
$W \subset Z$ completely contained in a piece of fixed order, say $\Gamma_i$. 
There are two options for $d_s^2$: either $W$ is collapsed first and then $Z$, 
or $Z$ is collapsed first and then $W$. 
In the first case the orientation induced is $(\ori_W)_Z = [o_k, o_h, ...]$ 
where $o_h$ and $o_k$ are created by the collapse of $W$ and $Z$ respectively. 
In the second case the induced orientation is $(\ori_Z)_W = [o_h, o_k, ...]$. 
In both cases $o_k$ corresponds with order $i+1$ and $o_h$ with order $i+2$. 
As in the case of $d_e$, the difference between these two is a sign and since all collapses in $d_s^2$ 
appear in this case we conclude $d_s^2 =0$.

Finally, to show $d_e d_s + d_s d_e=0$ consider an edge $e$ so that 
both of its vertices are not labeled points and $Z$ a semistable subset completely contained in a piece of a fixed order. 
Choosing a representative of the form $\ori = [e, ...]$ again yields two options: 
we can collaps $e$ and then $Z$, or viceversa. 
In the first case the orientation induced is $(\ori_e)_Z = [o_k, ...]$ and in the second 
the induced orientation is $(\ori_Z)_e = -[o_k,...]$ because of $\ori_Z = [o_k,e,...] = -[e,o_k,...]$.
\end{proof}

The previous theorem shows that \[ G^{g,n} = \left( \bigoplus_{k \ge 0} G_k^{g,n},d = d_e + d_s \right) \] is a chain complex.

\begin{figure}

\begin{center}
\begin{tikzpicture}[scale=1.3]
\fill (0,2) circle (2pt);
\fill (-1.732,-1) circle (2pt);
\fill (1.732,-1) circle (2pt);
\draw[thick] (0,2) -- (-1.732,-1) -- (1.732,-1) -- (0,2);

\fill (0,-1) circle (2pt);
\fill (-0.866,0.5) circle (2pt);
\fill (0.866,0.5) circle (2pt);
\draw[thick] (0,-1) -- (-0.866,0.5) -- (0.866,0.5) -- (0,-1);

\node at (0,0) {$\Theta$};
\node at (0,1) {$\Omega_1$};
\node at (-0.866,-0.5) {$\Omega_2$};
\node at (0.866,-0.5) {$\Omega_3$};
\end{tikzpicture}
\end{center}

\caption[Complex (0,3)]{The complex $G^{0,3}$.}  \label{simplecomplex}
\end{figure}
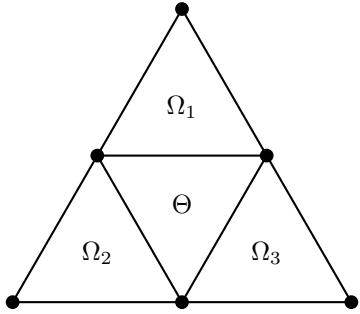

In Figure~\ref{simplecomplex} a geometric representation of $G^{0,3}$ is shown 
where $\Theta$ is the theta graph and the $\Omega_i$'s represent the three non-isomorphic ribbon graphs 
obtained from different labelings of the $\Omega$ graph. 
Their faces are also graphs in Figure~\ref{simplegraphs}. 
In this example the differential $d_s$ is trivially zero because a three labeled sphere 
does not degenerate to a singular surface. 
For $G^{0,4}$ the map $d_s$ is not zero but $d_s^2$ is trivially zero because a sphere with four labels cannot degenerate twice.

\subsection{Differences with previous work}

The chain complex $G^{g,n}$ contains the graph complex studied in \cite{convog} to formalize the ideas in \cite{kon:fncsg}. 
A few differences worth mentioning are the following.

In \cite[Section 4]{convog} the graph complex is denoted by $(P \overline{\mathcal O G})_S$ 
where $S$ is a surface of topological type $(g,n)$ generated by a graph, 
$P$ refers to ``primitive part'', which means that the graphs are connected, 
and $\mathcal O$ refers to ``operad'', 
it is taken to be the associated operad so that the complex is based on ribbon graphs. 
All graphs in \cite{convog} are contained in $G^{g,n}$. 
The additional non-semistable ribbon graphs in $G^{g,n}$ not included in $(P \overline{\mathcal O G})_S$ 
are the ones obtained from collapsing loops corresponding to labeled points. 
Clearly, semistable graphs (correponding to singular surfaces) represent new elements of the graph complex.

The orientation described in \cite{kon:fncsg} is further explored in \cite{convog} and \cite{lavo}. 
The orientation of a ribbon graph in the graph complex is a unit vector in $\det (\R E(\Gamma) \oplus H_1(\Gamma, \R))$. 
The $H_1(\Gamma, \R)$ term however is not altered by the differential defined in previous work 
because the graphs are not allowed to degenerate enought to change it's homology. 
The orientation used in $G^{g,n}$ is similar to the ``twisted orientation'' in \cite{lavo} (without the $H_1(\Gamma, \R)$ term). 
This makes sense because $G^{g,n}$ is based on the space $\MUC{g}{n}$ which is compact (\cite{zun}) and 
the twisted version of the orientation in \cite{lavo} corresponds with (co)homology with compact support. 

The analogue of $d_e$ in $(P \overline{\mathcal O G})_S$ is trivially zero when collapsing an edge that forms a loop in the graph. 
This is not the case in $G^{g,n}$ because such loops are allowed to degenerate since the moduli spaced being modeled is compact 
(this is the case of decorations on labeled points being allowed to degenerate to zero). 
The differential $d_s$ is new, but it is hinted in \cite{kon:itma} and \cite{loo}.
\medskip

To show that $G^{g,n}$ computes the homology of the moduli space $\MU{g}{n}$ 
it would be necessary as in \cite{convog} to describe an spectral sequence generated by the filtration 
induced by the orbicell decomposition in $\MUC{g}{n}$, 
and show that it converges to the expected homology.
\bigskip

\bibliographystyle{amsalpha}
\bibliography{GHME}

\providecommand{\bysame}{\leavevmode\hbox to3em{\hrulefill}\thinspace}
\providecommand{\MR}{\relax\ifhmode\unskip\space\fi MR }
\providecommand{\MRhref}[2]{%
  \href{http://www.ams.org/mathscinet-getitem?mr=#1}{#2}
}
\providecommand{\href}[2]{#2}
\begin{thebibliography}{KWZn12}

\bibitem[Cos09]{cos:pftft}
K.~J. Costello, \emph{The partition function of a topological field theory}, J.
  Topology \textbf{2} (2009), no.~4, 779--822, \texttt{math.QA/0509264}.

\bibitem[CV03]{convog}
J.~Conant and K.~Vogtmann, \emph{On a theorem of {K}ontsevich}, Algebr. Geom.
  Topol. \textbf{3} (2003), 1167--1224 (electronic). \MR{MR2026331
  (2004m:18006)}

\bibitem[HVZn10]{hvz}
E.~Harrelson, A.~Voronov, and J.~Z\'{u}\~{n}iga, \emph{Open-closed moduli
  spaces and related algebraic structures}, Letters in Mathematical Physics
  \textbf{94} (2010), no.~1.

\bibitem[Kon92]{kon:itma}
M.~Kontsevich, \emph{Intersection theory on the moduli space of curves and the
  matrix airy function}, Communications in Mathematical Physics \textbf{147}
  (1992).

\bibitem[Kon93]{kon:fncsg}
\bysame, \emph{Formal (non)-commutative symplectic geometry}, The Gelfand
  Mathematical Seminars, 1990--1992, Gelfand Math. Sem., Birkh\"auser Boston,
  Boston, MA, 1993, pp.~173--187.

\bibitem[KWZn12]{kwz}
R.~Kaufmann, B.~Ward, and J.~Z\'{u}\~{n}iga, \emph{The odd origin of
  {G}erstenhaber, {BV} and the master equation}, Preprint, August 2012,
  \texttt{arXiv:1208.5543[math.AT]}.

\bibitem[Loo95]{loo}
E.~Looijenga, \emph{Cellular decompositions of compactified moduli spaces of
  pointed curves}, The moduli space of curves (Texel Island, 1994), Progr.
  Math., vol. 129, Birkh\"auser Boston, Boston, MA, 1995, pp.~369--400.
  \MR{MR1363063 (96m:14031)}

\bibitem[LV08]{lavo}
A.~Lazarev and A.~Voronov, \emph{Graph homology: Koszul and verdier duality},
  Adv. Math. (2008), no.~6, 1878--1894.

\bibitem[Pen87]{pen87}
R.~Penner, \emph{The decorated teichm\"{u}ller space of punctured surfaces},
  Comm. Math. Phys. \textbf{113} (1987), no.~2, 299--339.

\bibitem[Str84]{strebel}
K.~Strebel, \emph{Quadratic differentials}, Ergebnisse der Mathematik und ihrer
  Grenzgebiete (3) [Results in Mathematics and Related Areas (3)], vol.~5,
  Springer-Verlag, Berlin, 1984. \MR{MR743423 (86a:30072)}

\bibitem[Zn15]{zun}
J.~Z\'{u}\~{n}iga, \emph{Compactifications of moduli spaces and cellular
  decompositions}, Algebraic and Geometric Topology (2015),
  \texttt{arXiv:math/0601130v1}.

\bibitem[Zvo04]{zvon}
D.~Zvonkine, \emph{{Strebel differentials on stable curves and Kontsevich's
  proof of Witten's conjecture}}, Jan 2004, \texttt{ArXiv:math/0209071v2}.

\bibitem[Zwi93]{zwi:csft}
B.~Zwiebach, \emph{Closed string field theory: quantum action and the
  {B}atalin-{V}ilkovisky master equation}, Nuclear Phys. B \textbf{390} (1993),
  no.~1, 33--152.

\bibitem[Zwi98]{zwi:oocst}
\bysame, \emph{Oriented open-closed string theory revisited}, Ann. Physics
  \textbf{267} (1998), no.~2, 193--248.

\end{thebibliography}


\end{document}